\documentclass{amsart}
\usepackage{amssymb}

\title{The Instructor's Guide to Real Induction}
\author{Pete L. Clark}
\address{Department of Mathematics \\ Boyd Graduate Studies Research Center \\ University 
of Georgia \\ Athens, GA 30602-7403 \\ USA}
\email{pete@math.uga.edu}

\thanks{The author thanks Fran\c cois Dorais, William G. Dubuque, Joel D. Hamkins, Niles Johnson and Iraj Kalantari for helpful conversations and/or 
pointers to the literature.}

\begin{document}

\maketitle


\newtheorem{lemma}{Lemma}
\newtheorem{prop}[lemma]{Proposition}
\newtheorem{cor}[lemma]{Corollary}
\newtheorem{thm}[lemma]{Theorem}
\newtheorem{example}[lemma]{Example}
\newtheorem{thm?}[lemma]{Theorem?}
\newtheorem{schol}[lemma]{Scholium}
\newtheorem{ques}{Question}
\newtheorem{conj}[lemma]{Conjecture}
\newtheorem*{mainthm}{Main Theorem}
\newtheorem{prob}[ques]{Problem}

\newcommand{\Z}{\mathbb{Z}}
\newcommand{\R}{\mathbb{R}}

\renewcommand{\dim}{\operatorname{dim}}
\newcommand{\LC}{\operatorname{LC}}
\newcommand{\AC}{\operatorname{AC}}
\newcommand{\ILC}{\operatorname{ILC}}
\newcommand{\IAC}{\operatorname{IAC}}
\renewcommand{\Z}{\mathbb{Z}}
\newcommand{\Q}{\mathbb{Q}}
\renewcommand{\R}{\mathbb{R}}
\newcommand{\C}{\mathbb{C}}
\newcommand{\F}{\mathbb{F}}
\newcommand{\N}{\mathbb{N}}
\newcommand{\ord}{\operatorname{ord}}
\newcommand{\den}{\operatorname{den}}
\newcommand{\ra}{\rightarrow}
\newcommand{\PP}{\mathcal{P}}
\newcommand{\diam}{\operatorname{diam}}
\newcommand{\Top}{\mathbb{T}}
\newcommand{\Bottom}{\mathbb{B}}

\section{Real Induction}
\noindent
In this article we discuss \textbf{real induction}, a proof technique 
similar to mathematical induction but applicable to proving statements indexed 
by an interval in $\R$.  This is not a new idea -- in fact 
among certain researchers in real analysis it is a commonplace -- but for such a natural and useful idea it is strangely poorly known by the mathematical community at large.  We also give a new inductive principle valid 
in a linearly ordered set which simultaneously generalizes real induction and transfinite induction (and thus also ordinary mathematical induction).  \\ \indent
This article is written primarily for an audience of teachers of undergraduate mathematics, especially honors calculus, real analysis and topology, hence the phrase ``instructor's guide''.  However I have for the most part backed away from explicit \emph{pedagogy}: that is, while I submit to you that there is interesting material here for several undergraduate courses, I make few specific recommendations as to how (or even whether) to teach it.  The ideas and techniques presented here should be of interest (especially?) to those who already have preferred proofs of these theorems.  This material could also 
serve as a source of challenges and enrichment for bright undergraduates -- and their instructors! -- and I have left some things for the reader to work out for herself as well as what I believe to be open problems.  

\subsection{``Induction is fundamentally discrete...''}
\textbf{} \\ \\ \noindent
Real induction is inspired by the principle of mathematical induction.  In broadest terms, the idea is to show that a statement holds for all 
real numbers in an interval by ``pushing from left to right''.  
\\ \\
In many circles it seems to be a truism that such a thing is not possible: how many times have you heard -- or said! -- that induction is fundamentally ``discrete'' and that, alas, inductive methods are not available for ``continuous variables''?  I held this idea myself until rather recently.  

\subsection{...is dead wrong!}
\textbf{} \\ \\ \noindent
Remarkably, this ``induction is fundamentally discrete'' idea has been refuted repeatedly in the literature, going back at least 93 years!  The earliest instance I know of is a 1919 note of Y.R. Chao \cite{Chao}. 
%
\begin{thm}(Chao)
\label{CHAOTHM}
Let $a \in \R$, and let $S \subset \R$.  Suppose that: \\
(CI1) $a \in S$.  \\
(CI2) There is $\Delta > 0$ such that $\forall x \in \R$, 
$x \in S \implies (x-\Delta,x+\Delta) \cap [a,\infty) \subset S$. \\
Then $[a,\infty) \subset S$.
\end{thm}
\noindent
Since this is not the formulation I want to discuss, I leave the proof to you, along with Chao's remark that the Archimedean nature of the 
ordering of $\R$ is being critically used here.  This Archimedean feature is 
absent in later formulations.
\\ \indent
But this is just the first -- if it actually is the first -- of many similar formulations of ``continuous induction''.  A literature search turned up 
the following papers, each of which introduces some form of
``continuous induction'', in many cases with no reference to past precedent: \cite{Khinchin23}, \cite{Perron}, \cite{Khinchin49}, \cite{Duren}, \cite{Ford}, \cite{Moss-Roberts}, \cite{Shanahan72}, \cite{Berekova}, \cite{Leinfelder},  \cite{Salat}, \cite{Dowek}, \cite{Kalantari}, \cite{Hathaway}.

\subsection{Real Induction} 
\textbf{} \\ \\ \noindent
Consider ``conventional'' mathematical induction.  To use it, one thinks in terms of \emph{predicates} -- i.e., statements $P(n)$ indexed by the natural numbers -- but the cleanest statement is in terms of \emph{subsets} of $\N$.  The same goes for real induction.
\\ \\
Let $a < b$ be real numbers.  We define a subset $S \subset [a,b]$ to be 
\textbf{inductive} if: \\ \\
(RI1) $a \in S$.  \\
(RI2) If $a \leq x < b$, then $x \in S \implies [x,y] \subset S$ for some 
$y > x$.  \\
(RI3) If $a < x \leq b$ and $[a,x) \subset S$, then $x \in S$.  
\begin{thm}(Real Induction)
\label{REALINDPROP}
\label{REALINDTHM}
For a subset $S \subset [a,b]$, the following are equivalent: \\
(i) $S$ is inductive.  \\
(ii) $S = [a,b]$.  
\end{thm}
\begin{proof}
(i) $\implies$ (ii): let $S \subset [a,b]$ be inductive.  Seeking a contradiction, suppose $S' = [a,b] \setminus S$ is nonempty, so $\inf S'$ exists and is finite. \\
Case 1: $\inf S' = a$.  Then by (RI1), $a \in S$, so by 
(RI2), there exists $y > a$ such that $[a,y] \subset S$, and thus $y$ 
is a greater lower bound for $S'$ then $a = \inf S'$: contradiction.  \\
Case 2: $a < \inf S' \in S$.  If $\inf S' = b$, then $S = [a,b]$.  Otherwise, by (RI2) there exists $y > \inf S'$ such that 
$[\inf S',y'] \subset S$, contradicting the definition of $\inf S'$.  \\
Case 3: $a < \inf S' \in S'$.  Then $[a,\inf S') \subset S$, so by 
(RI3) $\inf S' \in S$: contradiction! \\
(ii) $\implies$ (i) is immediate.
\end{proof}
\noindent
Theorem \ref{REALINDTHM} is due to D. Hathaway \cite{Hathaway} and, independently, to me.  But mathematically equivalent ideas have been around in the literature for a long time, some of which are \emph{much closer} to our formulation than the one of Chao given above.  
Especially, I acknowledge my indebtedness to \cite{Kalantari}.  \\ \indent
I hope I have been absolutely clear that the enunciation of an inductive principle for subintervals of $\R$ and its application to basic results of analysis has a lot of precedence in the literature. In writing this article, I am wholeheartedly agreeing with their approach and claiming that there is even more to be said.

\section{A Principle of Linearly Ordered Induction}
 \noindent
More generally, let $(X,\leq)$ be a linearly ordered set.  $X$ has at most one \textbf{top element} (resp. at most one \textbf{bottom element}); if such an element exists we denote it by $\Top$ (resp. $\Bottom$).  $X$ is 
\textbf{Dedekind complete} if every nonempty subset which is bounded above has a least upper bound, or, equivalently, if every nonempty subset which is bounded 
below has a greatest lower bound.
\\ \\
The preceding definitions are standard, but this one is new: a subset 
$S$ of a linearly ordered set $(X,\leq)$ is \textbf{inductive} if it satisfies all of the following: \\ \\
(IS1) There exists $a \in X$ such that $(-\infty,a] \subset S$. \\
(IS2) For all $x \in S$, either $x = 1$ or there exists $y > x$ such that $[x,y] \subset S$. \\
(IS3) For all $x \in S$, if $(-\infty,x) \in S$, then $x \in S$.

\begin{thm}
\label{MAINTHM}
\label{MTIOSTHM}
\label{ORDEREDINDTHM}
For a linearly ordered set $X$, TFAE:  \\
(i) $X$ is Dedekind complete.  \\
(ii) The only inductive subset of $X$ is $X$ itself.
\end{thm}
\begin{proof}
(i) $\implies$ (ii): Let $S \subset X$ be 
inductive.  Seeking a contradiction, we suppose $S' = X \setminus S$ 
is nonempty.  Fix $a \in X$ satisfying (IS1).  Then $a$ is a lower bound for $S'$, so by hypothesis $S'$ 
has an infimum, say $y$.  Any element less than $y$ is strictly less than every element of $S'$, so $(-\infty,y) \subset 
S$.  By (IS3), $y \in S$.  If $y = \Top$, then $S' = \{\Top\}$ or $S' = \varnothing$: both are contradictions.  So $y < \Top$, 
and then by (IS2) there exists $z > y$ such that $[y,z] \subset S$ and thus $(-\infty,z] \subset S$.  Thus $z$ is a 
lower bound for $S'$ which is strictly larger than $y$, contradiction. \\
(ii) $\implies$ (i): Let $T \subset X$ be nonempty and bounded below by 
$a$.  Let $S$ be the set of lower bounds for $T$.  Then $(-\infty,a] \subset S$, so $S$ satisfies (IS1). \\
Case 1: Suppose $S$ does not satisfy (IS2): there is $x \in S$ with no $y \in X$ such that $[x,y] \subset S$.  Since $S$ is downward closed, $x$ is the top element of $S$ and $x = \inf(T)$.  \\
Case 2: Suppose $S$ does not satisfy (IS3): there is $x \in X$ such that $(-\infty,x) \in S$ but $x \not \in S$, 
i.e., there exists $t \in T$ such that $t < x$.  Then also $t \in S$, so $t$ is the least element of $T$: in particular 
$t = \inf T$.  \\
Case 3: If $S$ satisfies (IS2) and (IS3), then $S = X$, so $T = \{\Top\}$ and $\inf T = \Top$.  
\end{proof} 
\noindent
If in Theorem \ref{MAINTHM} we take $X$ to be a closed subinterval of $\R$, we recover Real Induction and its variants as discussed above.  Taking 
$X$ to be any well-ordered set, we recover \textbf{transfinite induction} and, in particular, ``ordinary'' mathematical induction.
\\ \\
Since $\R$ is the unique Dedekind complete ordered 
field, we deduce:

\begin{cor}
\label{INDCOMPTHM}
In an ordered field $F$, the following are equivalent: \\
(i) $F = \R$. \\
(ii) For all $a < b \in F$, 
if $S \subset [a,b]$ satisfies (IS1), (IS2) and (IS3), then 
$S = [a,b]$.  
\end{cor}

\begin{prob} Characterize the inductive subsets of $[0,1] \cap \Q$.
\end{prob}

\section{Real Induction In Calculus}

\begin{thm}(Intermediate Value Theorem (IVT)) 
\label{INTTHM1}
\label{IVT}
Let $f: [a,b] \ra \R$ be a continuous 
function, and let $L$ be any number in between $f(a)$ and $f(b)$.  Then 
there exists $c \in [a,b]$ such that $f(c) = L$.
\end{thm}
\begin{proof}
It is easy to reduce the theorem to the following special case: if $f: [a,b] \ra \R \setminus \{0\}$ is continuous 
and $f(a) > 0$, then $f(b) > 0$. Put 
\[S = \{x \in [a,b] \ | \ f(x) > 0 \}, \]
 so $f(b) > 0$ iff $b \in S$.  We'll show $S = [a,b]$.  (RI1) By hypothesis, $f(a) > 0$, so $a \in S$.  (RI2) Let $x \in S$, $x < b$, so $f(x) > 0$.  Since $f$ is continuous at 
$x$, there exists $\delta > 0$ such that $f$ is positive on $[x,x+\delta]$, 
and thus $[x,x+\delta] \subset S$.  (RI3) Let $x \in (a,b]$ be such that $[a,x) \subset S$, i.e., $f$ is positive 
on $[a,x)$.  We claim that $f(x) > 0$.  Indeed, since $f(x) \neq 0$, the only other possibility is $f(x) < 0$, but if so, then by continuity there would exist $\delta > 0$ such that $f$ is negative on 
$[x-\delta,x]$, i.e., $f$ is both positive and negative at each point of $[x-\delta,x]$: contradiction!
\end{proof}

\begin{thm}(Extreme Value Theorem (EVT)) \\
\label{INTTHM2}
\label{EVT}
Let $f: [a,b] \ra \R$ be continuous. Then: \\
a) The function $f$ is bounded.  \\
b) The function $f$ attains its maximum and minimum values.
\end{thm}
\begin{proof} a) Let $S = \{x \in [a,b] \ | \ f: [a,x] \ra \R \text{ is bounded} \}$.  (RI1): Evidently $a \in S$.  (RI2): Suppose $x \in S$, so that $f$ is bounded on $[a,x]$.  But then $f$ is continuous at $x$, so is bounded near 
$x$: for instance, there exists $\delta > 0$ such that for all $y \in [x-\delta,x+\delta]$, $|f(y)| \leq |f(x)| + 1$.  
So $f$ is bounded on $[a,x]$ and also on $[x,x+\delta]$ and thus on $[a,x+\delta]$.  (RI3): Suppose $x \in (a,b]$ and $[a,x) \subset S$.  Since $f$ is 
continuous at $x$, there exists $0 < \delta < x-a$ such that $f$ is bounded on $[x-\delta,x]$.  Since $a < x-\delta < x$, $f$ is bounded 
on $[a,x-\delta]$, so $f$ is bounded on $[a,x]$. \\
b) Let $m = \inf f([a,b])$ and $M = \sup f([a,b])$.  By part a) we have 
\[-\infty < m \leq M < \infty. \] We want to show that there exist 
$x_m, x_M \in [a,b]$ such that $f(x_m) = m$, $f(x_M) = M$, i.e., that the infimum and supremum are actually attained as values of $f$.  Suppose that there does not exist $x \in [a,b]$ with $f(x) = m$: then $f(x) > m$ for all $x \in [a,b]$ 
and the function $g_m: [a,b] \ra \R$ by $g_m(x) = \frac{1}{f(x) - m}$ is 
defined and continuous.  By the result of part a), $g_m$ is bounded, but this is absurd: by definition of the infimum, $f(x) - m$ takes values less than $\frac{1}{n}$ for any $n \in \Z+$ and thus $g_m$ takes values greater than 
$n$ for any $n \in \Z^+$ and is accordingly unbounded.  So indeed there must exist $x_m \in [a,b]$ such that $f(x_m) = m$.  Applying this argument to $-f$, 
we see that $f$ also attains its maximum value.  
\end{proof}

\noindent
Let $f: I \ra \R$.  For $\epsilon,\delta > 0$, let us say that $f$ 
is $(\epsilon,\delta)$-UC on $I$ if for all $x_1,x_2 \in I$, $|x_1-x_2| < \delta \implies |f(x_1) - f(x_2)| < \epsilon$.  By definition, $f: I \ra \R$ 
is \textbf{uniformly continuous} if for all $\epsilon > 0$, there is $\delta > 0$ 
such that $f$ is $(\epsilon,\delta)$-UC on $I$.

\begin{lemma}(Covering Lemma)
\label{PATCHINGLEMMA}
Let $a < b < c < d$ be real numbers, and let $f: [a,d] \ra \R$.  Suppose 
that for real numbers $\epsilon_1,\delta_1,\delta_2 > 0$, \\
$\bullet$ $f$ is $(\epsilon,\delta_1)$-UC on $[a,c]$ and \\
$\bullet$ $f$ is $(\epsilon,\delta_2)$-UC on $[b,d]$.  \\
Then $f$ is $(\epsilon,\min(\delta_1,\delta_2,c-b))$-UC on $[a,b]$.
\end{lemma}
\begin{proof}
Left to the reader.
\end{proof}

\begin{thm}(Uniform Continuity Theorem) 
\label{INTTHM3} 
\label{UNIFORMCONTINUITYTHM}
\label{UCT}
Let $f: [a,b] \ra \R$ be continuous.  Then $f$ is uniformly 
continuous on $[a,b]$.
\end{thm}
\begin{proof}
For $\epsilon > 0$, let $S(\epsilon)$ be the set of $x \in [a,b]$ such that there exists $\delta > 0$ such that $f$ is $(\epsilon,\delta)$-UC on $[a,x]$.  
To show that $f$ is uniformly continuous on $[a,b]$, it suffices to show that $S(\epsilon) = [a,b]$ for all $\epsilon > 0$.  We will 
show this by Real Induction.  (RI1): Trivially $a \in S(\epsilon)$.  (RI2): Suppose $x \in S(\epsilon)$, so there exists 
$\delta_1 > 0$ such that $f$ is $(\epsilon,\delta_1)$-UC on $[a,x]$. Moreover, since $f$ is continuous at $x$, there is $\delta_2 > 0$ 
such that for all $c \in [x,x+\delta_2]$, $|f(c) - f(x)| < \frac{\epsilon}{2}$.  Thus for all $c_1, c_2 \in [x-\delta_2,x+\delta_2]$, 
\[ |f(c_1) - f(c_2)| = |f(c_1) - f(x) + f(x) - f(c_2)| \leq |f(c_1) - f(x)| 
+ |f(c_2) - f(x)| < \epsilon, \]
so $f$ is $(\epsilon,\delta_2)$-UC on $[x-\delta_2,x+\delta_2]$.
Apply the Covering Lemma to $f$ with $a < x-\delta_2 < x < x + \delta_2$: we conclude $f$ is $(\epsilon,\min(\delta,\delta_2,x-(x-\delta_2))) = (\epsilon,\min(\delta_1,\delta_2))$-UC on $[a,x+\delta_2]$.  Thus $[x,x+\delta_2] \subset S(\epsilon)$.  (RI3): Suppose $[a,x) \subset S(\epsilon)$.  As above, since $f$ is continuous at $x$, there is $\delta_1 > 0$ such that $f$ is $(\epsilon,\delta_1)$-UC 
on $[x-\delta_1,x]$.  Since $x - \frac{\delta_1}{2} < x$, by hypothesis 
there exists $\delta_2$ such that $f$ is $(\epsilon,\delta_2)$-UC on 
$[a,x-\frac{\delta_1}{2}]$.  Apply the Covering Lemma to $f$ 
with $a < x - \delta_1 < x - \frac{\delta_1}{2} < x$: we conclude 
$f$ is $(\epsilon,\min(\delta_1,\delta_2,x-\frac{\delta_1}{2} - (x-\delta_1))) = 
(\epsilon,\min(\frac{\delta_1}{2},\delta_2))$-UC on $[a,x]$.  Thus 
$x \in S(\epsilon)$.  
\end{proof}

\begin{thm}
\label{CONTINTTHM}
\label{DIT}
Let $f: [a,b] \ra \R$ be continuous.  Then $f$ is Darboux integrable.
\end{thm}
\begin{proof}
By Darboux's Criterion \cite[Thm. 13.2]{Spivak}, $f$ is Darboux integrable iff 
for all $\epsilon > 0$, there exists a partition $\PP$ of $[a,b]$ such that 
$U(f,\PP) - L(f,\PP) < \epsilon$.  It is convenient to prove instead the following equivalent statement: for every $\epsilon > 0$, there exists a partion $\PP$ of $[a,b]$ such that $U(f,\PP) - L(f,\PP) < (b-a)\epsilon$.  \\ \indent
Fix $\epsilon > 0$, and let $S(\epsilon)$ be the set of $x \in [a,b]$ 
such that there exists a partition $\PP_x$ of $[a,b]$ with $U(f,\PP_x) - L(f,\PP_x) < \epsilon$.  We want to show $b \in S(\epsilon)$, so it suffices to show $S(\epsilon) = [a,b]$.  In fact it is necessary and sufficient: observe that if $x \in S(\epsilon)$ and $a \leq y \leq x$, then also $y \in S(\epsilon)$. We will show $S(\epsilon) = [a,b]$ by Real Induction.  \\
(RI1) The only partition of $[a,a]$ is $\PP_a = \{a\}$, and for this partition we have $U(f,\PP_a) = L(f,\PP_a) = f(a) \cdot 0 = 0$, so $U(f,\PP_a) - L(f,\PP_a) = 0 < \epsilon$.  \\
(RI2) Suppose that for $x \in [a,b)$ we have $[a,x] \subset S(\epsilon)$.  
We must show that there is $\delta > 0$ such that $[a,x+\delta] \subset S(\epsilon)$, and by the above observation it is enough to find $\delta > 0$ 
such that $x+\delta \in S(\epsilon)$: we must find a partition $\PP_{x+\delta}$ of $[a,x+\delta]$ such that $U(f,\PP_{x+\delta}) - L(f,\PP_{x+\delta}) < (x+\delta-a)\epsilon)$.  Since $x \in S(\epsilon)$, 
there is a partition $\PP_x$ of $[a,x]$ with $U(f,\PP_x) - L(f,\PP_x) < 
(x-a)\epsilon$.  Since $f$ is continuous at $x$, we can make the difference between the supremum and the infimum of $f$ as small as we want by taking a sufficiently small interval around $x$: i.e., there is $\delta > 0$ such that $\sup(f,[x,x+\delta]) - \inf(f,[x,x+\delta]) < \epsilon$.  Now take the smallest partition 
of $[x,x+\delta]$, namely $\PP' = \{x,x+\delta\}$.  Then $U(f,\PP') - L(f,\PP') = 
(x+\delta-x) (\sup(f,[x,x+\delta]) - \inf(f,[x,x+\delta])) < \delta \epsilon$.  
Thus if we put $\PP_{x+\delta} = \PP_x \cup \PP'$ and use the fact that upper / lower sums add when split into subintervals, we have 
\[ U(f,\PP_{x+\delta}) - L(f,\PP_{x+\delta}) = U(f,\PP_x) + U(f,\PP') - 
L(f,\PP_x) - L(f,\PP') \]
\[ = U(f,\PP_x) - L(f,\PP_x) + 
U(f,\PP') - L(f,\PP') < (x-a) \epsilon + \delta \epsilon = (x+\delta -a)\epsilon. \]
(RI3) Suppose that for $x \in (a,b]$ we have $[a,x) \subset S(\epsilon)$.  
We must show that $x \in S(\epsilon)$.  The argument for this is the same as for (RI2) except we use the interval $[x-\delta,x]$ instead of $[x,x+\delta]$.  Indeed: since $f$ is continuous at $x$, there exists 
$\delta > 0$ such that $\sup(f,[x-\delta,x]) - \inf(f,[x-\delta,x]) < \epsilon$.  
Since $x-\delta < x$, $x-\delta \in S(\epsilon)$ and thus there exists a 
partition $\PP_{x-\delta}$ of $[a,x-\delta]$ such that $U(f,\PP_{x-\delta}) = 
L(f,\PP_{x-\delta}) = (x-\delta-a)\epsilon$.  Let $\PP' = \{x-\delta,x\}$ and 
let $\PP_x = \PP_{x-\delta} \cup \PP'$.  Then 
\[ U(f,\PP_x) - L(f,\PP_x) = U(f,\PP_{x-\delta}) + U(f,\PP') - 
(L(f,\PP_{x-\delta}) + L(f,\PP')) \]
\[ = (U(f,\PP_{x-\delta}) - L(f,\PP_{x-\delta})) + \delta (\sup(f,[x-\delta,x]) - \inf(f,[x-\delta,x])) \]
\[ < (x-\delta - a)\epsilon + \delta \epsilon = (x-a) \epsilon. \]
\end{proof}
\noindent
Remark: Spivak's text \cite{Spivak} relegates uniform continuity to an appendix, 
which creates a slight quandary when it comes to proving Theorem \ref{DIT}.  He gives the standard proof anyway \cite[Thm. 13.3]{Spivak}, but also \cite[pp. 292-293]{Spivak} an alternative argument establishing equality 
of the upper and lower integrals by differentiation.  This method goes back at least to M.J. Norris \cite{Norris}.  The Real Induction proof given above 
seems quite different from both of these.  

\begin{thm}(Bolzano-Weierstrass)
\label{BWTHM}
An infinite subset of $[a,b]$ has a limit point.
\end{thm}
\begin{proof}
Let $\mathcal{A} \subset [a,b]$, and let $S$ be the set of $x$ in $[a,b]$ such that \emph{if} $\mathcal{A} \cap [a,x]$ is infinite, it has a limit point.  It suffices to show $S = [a,b]$, which we will do by Real Induction.  (RI1) is clear.  (RI2) Suppose $x \in [a,b) \cap S$.  If $\mathcal{A} \cap [a,x]$ is infinite, then it has a limit point and hence so does $\mathcal{A} \cap [a,b]$: thus $S = [a,b]$.  If for some $\delta > 0$, $\mathcal{A} \cap [a,x+\delta]$ is finite, then $[x,x+\delta] \subset S$.  Otherwise $\mathcal{A} \cap [a,x]$ is finite but $\mathcal{A} \cap 
[a,x+\delta]$ is infinite for all $\delta > 0$, and then $x$ is a limit 
point for $\mathcal{A}$ and $S = [a,b]$ as above.  (RI3) If $[a,x) \subset S$, 
then: either $\mathcal{A} \cap [a,y]$ is infinite for some $y < x$, so 
$x \in S$; or $\mathcal{A} \cap [a,x]$ is finite, so $x \in S$; or $\mathcal{A} \cap [a,y]$ is finite for all $y < x$ and $\mathcal{A} \cap [a,x]$ 
is infinite, so $x$ is a limit point of $\mathcal{A} \cap [a,x]$ and $x \in S$.
\end{proof} 
\noindent
Remark: The standard proofs of Theorem \ref{BWTHM} use monotone subsequences, a 
dissection / nested intervals argument, or deduce the result from compactness 
of $[a,b]$.  The Real Induction proof given above seems to be new.

\section{Connectedness and Compactness}
\noindent
The preceding applications of Real Induction were at the honors calculus 
level.  Next we consider applications to undergraduate real 
analysis and topology, specifically to results involving connectedness and compactness of subsets of a linearly ordered topological space.  We assume 
familiarity with the following standard results.

\begin{prop}
\label{EASYCONNECTEDPROP}
\label{4.0}
Let $f: X \ra Y$ be a surjective continuous map of topological spaces.  If $X$ is connected, then so is $Y$. 
\end{prop}

\begin{prop}
\label{4.1}
Let $f: X \ra Y$ be a surjective continuous map of topological spaces.  If $X$ is compact, then so is $Y$.
\end{prop}
\noindent
In real analysis, Theorem \ref{IVT} is \emph{analyzed} as the conjunction of Proposition \ref{4.0} and the deep fact that the interval $[a,b] \subset \R$ is connected.  Similarly, Theorem \ref{EVT} is analyzed as the conjunction of the easy Proposition \ref{4.1} and the deep fact that the interval $[a,b] \subset \R$ is compact.   Real Induction is very well-suited to proving connectedness and compactness results in ordered spaces.  

\begin{thm}
\label{CONNECTEDTHM}
\label{4.2}
The interval $[a,b]$ is connected.
\end{thm}
\begin{proof} 
Suppose $[a,b] = A \cup B$, with $A$ and $B$ open and closed, and $A \cap B = \emptyset$.  We assume $a \in A$ and prove by Real Induction that $A = [a,b]$: (RI1) is immediate, (RI2) holds because $A$ is open, and (RI3) holds because $A$ is closed.  We're done! 
\end{proof}
\noindent
The connectedness of $[a,b]$ follows easily from the Intermediate Value Theorem: a separation of $[a,b]$ yields a surjective 
continuous function $f: [a,b] \ra \{ 0,1\}$.  But the proof of Theorem \ref{4.2} is shorter and easier than that of Theorem \ref{IVT}.   
\\ \\
This circle of ideas can be pushed further using the concepts and results of $\S 2.3$.  Namely, let $(X,\leq)$ be a linearly ordered set.  A \textbf{bounded open interval} in a linearly ordered set 
$X$ is an interval of the form $(a,b)$ or $[\Bottom,b)$ (if $X$ has a bottom element) 
or $(a,\Top]$ (if $X$ has a top element).  The bounded 
open intervals form a base for the \textbf{order topology} on $X$.  A linearly ordered set $X$ is \textbf{dense} if for all $a < b \in X$, 
there exists $c$ with $a < c < b$.

\begin{lemma}
\label{LEMMA1}
Let $S$ be a subset of a Dedekind complete ordered set $(X,\leq)$.  If $S$ is closed in $X$ in the order topology, then $(S,\leq)$ is Dedekind complete.
\end{lemma}
\begin{proof} This is straightforward and left to the reader.
\end{proof}
\noindent
The converse of Lemma \ref{LEMMA1} does not hold, 
even for subsets of $\R$: indeed, \emph{any} subinterval of $\R$ is Dedekind complete.

\begin{thm}
\label{THM1}
For a linearly ordered set $X$, TFAE: \\
(i) $X$ is dense and Dedekind complete.  \\
(ii) $X$ is connected in the order topology.  
 \end{thm}
\begin{proof}
(i) $\implies$ (ii): Step 1: We suppose $0 \in X$.  Since $X$ is dense, a subset $S \subset X$ 
which contains $0$ and is both open and closed in the order topology is inductive.  Since $X$ is Dedekind complete, by Theorem \ref{MAINTHM}, $S = X$.  This shows $X$ is connected!  \\
Step 2: We may assume $X \neq \varnothing$ and 
choose $a \in X$.  By Lemma \ref{LEMMA1}, Step 1 applies to show $[a,\infty)$ connected.  A similar downward induction argument shows $(-\infty,a]$ is connected.  Since $X = (-\infty,a] \cup [a,\infty)$ and $(-\infty,a] \cap 
[a,\infty) \neq \varnothing$, $X$ is connected. \\
(ii) $\implies$ (i): If the order is 
not dense, there are $a < b$ in $X$ with $[a,b] = \{a,b\}$, so $A = (-\infty,a], \ B = [b,\infty)$ is a separation of $X$.  Suppose we have $S \subset X$, nonempty, bounded below by $a$ and with no infimum.  Let $L$ be the set of lower bounds for $S$, and put $U = \bigcup_{b \in L} 
(-\infty,b)$, so $U$ is open and $U < S$.  We have $a \neq \inf(S)$, so $a \in U$, and thus $U \neq \varnothing$.  If $x \not \in U$, then 
$x \geq L$ and, indeed, since $L$ has no top element, $x > L$, so there exists $s \in S$ such that $s < x$.  Since the order is dense there is $y$ with $s < y < x$, and then the entire open set $(y,\infty)$ lies in the complement of $U$.  Thus $U$ is also closed.  Since $X$ is connected, $U = X$, contradicting $U < S$. 
\end{proof}

\begin{cor}
\label{CONNECTEDCONVERSE}
Let $(F,+,\cdot,<)$ be an ordered field.  The following are equivalent: \\
(i) $F$ is Dedekind complete -- i.e., $F \cong \R$.  \\
(ii) Every closed interval $[a,b]$ of $F$ is connected in the order topology.
\end{cor}
\noindent
Now we turn to compactness.

\begin{thm}(Heine-Borel)
\label{COMPACTTHM}
\label{4.3}
\label{HEINEBOREL}
The interval $[a,b]$ is compact.
\end{thm}
\begin{proof} For an open covering $\mathcal{U} = \{U_i\}_{i \in I}$ of $[a,b]$, let \[S = \{x \in [a,b] \ | \ \mathcal{U} \cap [a,x] \text{ has a finite subcovering} \}. \]  We prove $S = [a,b]$ by Real Induction. (RI1) is clear.  (RI2): If $U_1,\ldots, U_n$ covers $[a,x]$, then some $U_i$ contains $[x,x+\delta]$ for some $\delta > 0$. (RI3): if $[a,x) \subset S$, let $i_x \in I$ be such that $x \in U_{i_x}$, and let $\delta > 0$ be such that $[x-\delta,x] \in U_{i_x}$.  Since 
$x-\delta \in S$, there is a finite $J \subset I$ with $\bigcup_{i \in J} U_i \supset [a,x-\delta]$, so $\{U_i\}_{i \in J} \cup U_{i_x}$ covers $[a,x]$. 
\end{proof}
\noindent
Remark: The Extreme Value Theorem implies the compactness of 
$[a,b]$, a claim we leave to the reader.  (But again, we prefer the 
proof of Theorem \ref{4.3}).
\\ \\
These ideas can be used to characterize compactness in linearly ordered topological spaces.  We say a linearly ordered set $(X,\leq)$ is \textbf{complete} 
if every subset has a least upper bound; equivalently, if every subset has an 
greatest lower bound.  $X$ is complete iff it is Dedekind complete and has 
top and bottom elements.  

\begin{thm}
\label{THM2}
For a nonempty linearly ordered set $X$, TFAE: \\
(i) $X$ is complete. \\
(ii) $X$ is compact in the order topology.   
\end{thm}

\begin{proof}
(i) $\implies$ (ii): Let $\mathcal{U} = \{U_i\}_{i \in I}$ be an open covering of $X$.  Let $S$ be the set of $x \in X$ such that the covering $\mathcal{U} \cap [0,x]$ of $[0,x]$ admits a finite subcovering.  $0 \in S$, so $S$ satisfies (IS1).  Suppose $U_1 \cap [0,x],\ldots, U_n \cap [0,x]$ covers $[0,x]$.  If there exists $y \in X$ such that $[x,y] = \{x,y\}$, then adding to the covering any element $U_y$ containing $y$ gives a finite covering of $[0,y]$.  Otherwise
some $U_i$ contains $x$ and hence also $[x,y]$ for some $y > x$.  So $S$ satisfies (IS2).  If we have a finite covering of $[0,x)$, then we get a finite covering of $[0,x]$ just by adding in any element $U_x$ which contains $x$, so $S$ satisfies (IS3).  Thus $S$ 
is an inductive subset of the Dedekind complete ordered set $X$ and $S = X$.  
So $1 \in S$ and the covering has a finite subcovering.  \\
(ii) $\implies$ (i): For each $x \in X$ there is a bounded open interval $I_x$ containing $x$.  If $X$ is compact, $\{I_x\}_{x \in X}$ has a finite subcovering, so $X$ is bounded, 
i.e., has $\Bottom$ and $\Top$.  Let $S \subset X$.  Since $\inf \varnothing = \Top$, we may assume $S \neq \varnothing$.  Since $S$ has an infimum iff $\overline{S}$ does, we may assume $S$ is closed 
and thus compact.  Let $L$ be the set of lower bounds for $S$.  For each $(b,s) \in L \times S$, consider the closed interval $C_{b,s} := [b,s]$.  For any finite subset $\{(b_1,s_1),\ldots,(b_n,s_n)\}$ of $L \times S$, $\bigcap_{i=1}^N [b_i,s_i] \supset [\max b_i, \min s_i] \neq \varnothing$.  
Since $S$ is compact there is  $y \in \bigcap_{L \times S} [b,s]$ and 
then $y = \inf S$.
\end{proof}  
\begin{cor}(Generalized Heine-Borel)
\label{HEINEBOREL}
a) For a linearly ordered set $X$, TFAE: \\
(i) $X$ is Dedekind complete. \\
(ii) A subset $S$ of $X$ is compact in the order topology iff it is closed and bounded.  \\
b) For an ordered field $F$, the following are equivalent: \\
(i) Every closed, bounded interval $[a,b] \subset F$ is compact.  \\
(ii) $F = \R$.
\end{cor}
\begin{proof}
a) (i) $\implies$ (ii): A compact subset of any ordered space is closed and bounded.  Conversely, if $X$ is Dedekind complete and $S \subset X$ is 
closed and bounded, then by Lemma \ref{LEMMA1}, $S$ is complete and then by 
Theorem \ref{THM2}, $S$ is compact.  \\
(ii) $\implies$ (i): If $S \subset X$ is nonempty and bounded above, let 
$a \in S$.  Then $S' = S \cap [a,\infty)$ is bounded, so $\overline{S'}$ 
is compact and thus $\overline{S'}$ is complete by Theorem \ref{THM2}.  The least upper bound of $\overline{S'}$ is also the least upper bound of $S$.  \\
b) This follows immediately.  
\end{proof}
\noindent


\section{Still More Real Induction}

\subsection{Around the Mean Value Theorem}
\textbf{} \\ \\ \noindent
For an interval $I \subset \R$, let us denote by $I^{\circ}$ the interior of $I$.  

%
\begin{thm}(Mean Value Theorem)
\label{MVT}
Let $f: I \ra \R$ be continuous and differentiable on $I^{\circ}$.  For any 
$a < b \in I$, there exists $c \in I^{\circ}$ such that $\frac{f(b)-f(a)}{b-a} = f'(c)$.
\end{thm}
\noindent
Many authors -- e.g. \cite{Bers67}, \cite{Cohen67}, \cite{Tucker97}, \cite{Koliha09} -- have advocated replacing Theorem \ref{MVT} with an alternative that is (somehow) more appealing / intuitive to calculus students, 
and especially with one of the following results.
\begin{cor}(Mean Value Inequality)
\label{MVCOR1}
Let $f: I \ra \R$ be differentiable.  Suppose that there exists $M \in \R$ 
such that for all $x \in I$, $f'(x) \geq M$.  Then for all $x < y \in \R$, $f(y) - f(x) \geq M (y-x)$.
\end{cor}

\begin{cor}(Weakly Increasing Function Theorem)
\label{MVCOR2}
Let $f: I \ra \R$ be differentiable, and suppose that for all $x \in I$, $f'(x) \geq 0$.  Then $f$ 
is \emph{weakly increasing} on $I$: for all $x,y \in I$, $x \leq y \implies 
f(x) \leq f(y)$.  
\end{cor}

\begin{cor}(Increasing Function Theorem)
\label{MVCOR3}
Let $f: I \ra \R$ be differentiable, and suppose that for all $x \in I$, 
$f'(x) > 0$.  Then $f$ is \emph{increasing} on $I$: for all $x,y \in I$, 
$x < y \implies f(x) < f(y)$.
\end{cor}

\begin{cor}(Racetrack Principle)
\label{MVCOR4}
Let $f,g: [a,b]  \ra \R$ be differentiable.  If $g'(x) \leq f'(x)$ for all $x \in [a,b]$, 
then $g(x) - g(a) \leq f(x) - f(a)$.  
\end{cor}

\begin{cor}(Zero Velocity Theorem)
\label{MVCOR5}
Let $f: I \ra \R$ be differentiable, and suppose that for all $x \in I^{\circ}$, 
$f'(x) = 0$.  Then $f$ is constant.
\end{cor}
\noindent
Corollaries \ref{MVCOR1}, \ref{MVCOR2}, \ref{MVCOR3} and \ref{MVCOR4} are 
equivalent (in the sense that each immediately implies all the others), are 
all implied by Theorem \ref{MVT}, and all imply Corollary \ref{MVCOR5}.  
\\ \indent I have no pedagogical misgivings about Theorem \ref{MVT}.  But there may be a mathematical distinction to be made between Theorem 
\ref{MVT} and Corollaries \ref{MVCOR1}, \ref{MVCOR2}, \ref{MVCOR3}, \ref{MVCOR4}: the first cannot -- so far as I know! -- be proved directly by Real Induction, while the others can.  We will be content to give a Real Induction proof of Corollary \ref{MVCOR1}.
\begin{proof}
Let $a < b \in I$.  Fix $\epsilon > 0$, and define 
\[ S_{\epsilon} = \{x \in [a,b] \ | \ f(x) -f(a) \geq (M-\epsilon)(x-a) \}. \]
We will show that $S_{\epsilon} = [a,b]$.  Then $b \in S_{\epsilon}$, so $f(b) -f(a) \geq (M-\epsilon)(b-a)$.  Since this 
holds for all $\epsilon > 0$, we deduce $f(b) - f(a) \geq  M(b-a)$.  (RI1) is immediate.  (RI2) Since $f'(x) \geq M$, there is $\delta > 0$ such that $y \in [x,x+\delta] \implies \frac{f(y) - f(x)}{y-x} \geq M - \epsilon$.  Thus $f(y) - f(a) = (f(y) - f(x)) + (f(x) - f(a)) \geq  (M-\epsilon)(y-x) + (M-\epsilon)(x-a) = (y-a)(M-\epsilon),$
so $[x,x+\delta] \subset S_{\epsilon}$.  (RI3) Suppose that for $x \in (a,b]$ we have $[a,x) \subset S_{\epsilon}$.  As above, since $f'(x) \geq M$, there exists $\delta > 0$ such that $y \in [x-\delta,x] \implies \frac{f(x)-f(y)}{x-y} \geq M - \epsilon$.  Thus \[f(x) - f(a) = (f(x) - f(y)) + (f(y) - f(a)) \] \[ \geq (M-\epsilon)(x-y) + (M-\epsilon)(y-a) = 
(x-a)(M-\epsilon), \]
so $x \in S_{\epsilon}$.  
\end{proof}

\begin{prob}
Give either a direct proof of the Mean Value Theorem by Real Induction or a convincing explanation for why such a proof cannot exist.
\end{prob}

\subsection{Some Real Induction Proofs For the Reader}
\textbf{} \\ \\ \noindent
Here are more results amenable to Real Induction.  The proofs are left to you.

\begin{thm}(Cantor Intersection Theorem)
\label{CANTORINTTHM}
Let $\{F_n\}_{n=1}^{\infty}$ be a decreasing sequence of closed 
subsets of $[a,b]$.  Put $F = \bigcap_n F_n$.  Then \emph{either} 
$F \neq \varnothing$ or there exists $n \in \Z^+$ such that $F_n = \varnothing$.
\end{thm}

\begin{thm}
\label{LEBESGUENUMBERTHM}
Any open covering of $[a,b]$ admits a Lebesgue number.
\end{thm}

\begin{thm}(Dini's Lemma)
Let $\{f_n\}_{n=1}^{\infty}$ be a sequence of continuous real-valued functions on the interval $[a,b]$ which is \emph{pointwise decreasing}: for all $x \in [a,b]$ and all $n \in \Z^+$, $f_{n+1}(x) \leq f_n(x)$.  If $f: [a,b] \rightarrow \R$ is continuous and $f_n \ra f$ pointwise, then $f_n \ra f$ uniformly.
\end{thm}

\begin{thm}(Arzel\`a-Ascoli) \textbf{} \\
Let $\{f_n\}_{n=1}^{\infty}$ be a sequence of continuous functions on $[a,b]$ such that: \\
(i) There is $M \in \R$ such that for all $n \in \Z^+$ and all $x \in [a,b]$, $|f_n(x)| \leq M$, and \\
(ii) For all $x \in [a,b]$ and all $\epsilon > 0$, there exists 
$\delta > 0$ such that if $|x-y| < \delta$, then for all $n \in \Z^+$, 
$|f_n(x) - f_n(y)| < \epsilon$. \\
Then there is a subsequence $\{f_{n_k}\}$ which is uniformly 
convergent on $[a,b]$.
\end{thm}
\noindent
 Now that you are aware of 
the method, I invite you explore its further uses.  
\begin{prob}
Find other theorems which can be proved via Real Induction.
\end{prob}

\end{document}